\documentclass[11pt]{article}%
\usepackage{amssymb}
\usepackage{amsmath}
\usepackage{mathrsfs}
\usepackage{amsthm}
\usepackage{graphicx}
\usepackage{amsfonts}
\usepackage{epstopdf}
\usepackage[pagewise]{lineno}%
\providecommand{\U}[1]{\protect\rule{.1in}{.1in}}
\newcounter{subeqn} 
\makeatletter
\@addtoreset{subeqn}{equation}
\makeatother

\newtheorem{theorem}{Theorem}

\newtheorem{lemma}{Lemma}
\newtheorem{remark}{Remark}

\begin{document}
	
	\title{On the stability of a double porous elastic system with visco-porous dampings}
	\author{Ahmed KEDDI$^{(1)}$, Aicha NEMSI$^{2}$ \& Abdelfeteh FAREH$^{(2)}$\\{\small $^{1}$ Laboratory of Mathematics, Modelisation and Application
			(LAMMA)}\\{\small Department of Mathematics and computer sciences, University of Adrar,
			Adrar, Algeria} \\
	{\small $^{2}$ Laboratory of operators theory and PDEs: Foundations and applications}
			\\{\small {Faculty of Exact sciences, University of El Oued,
				P.B. 789, El Oued 39000, Algeria.}}}
	\date{}
	\maketitle
	 \emph{2020 Mathematics Subject Classification}: 35B35; 35B40; 35L15; 35Q74; 74F10; 93D05\\

\emph{Key words and phrases}: double porosity, well-posedness, exponential decay, lack of exponential decay.





\date{}

\maketitle

\renewcommand{\thefootnote}{}

\footnote{emails:~ahmedkeddi@univ-adrar.dz~aichanemsi@gmail.com~ \\
	Corresponding authors:farehabdelf@gmail.com. }
%

\renewcommand{\thefootnote}{\arabic{footnote}}
\setcounter{footnote}{0}


\begin{abstract}
	In this paper we consider a one dimensional elastic system with double porosity structure and with frictional damping in both porous equations. We introduce two stability numbers $\chi_{0}$ and $\chi_{1}$ and prove that the solution of the system decays exponentially provided that $\chi_{0}=0$ and $\chi_{1}\neq0.$ Otherwise,  we prove the lack of exponential decay.
Our results improve the results of \cite{Bazarra} and \cite{Nemsi}.

\end{abstract}

\section{Introduction}

In this paper, we are concerned with the following system%
\begin{equation}
	\left\{
	\begin{array}
		[c]{lc}%
		\rho u_{tt}=\mu u_{xx}+b\varphi_{x}+d\psi_{x}& \mbox{in}~(0,\pi)\times\mathbb{R}_+,\\
		\kappa_{1}\varphi_{tt}=\alpha\varphi_{xx}+\beta\psi_{xx}-bu_{x}-\alpha
		_{1}\varphi-\alpha_{3}\psi-\tau_{1}\varphi_{t}-\tau_{2}\psi_{t} &
		\mbox{in}~(0,\pi)\times\mathbb{R}_+,\\
		\kappa_{2}\psi_{tt}=\beta\varphi_{xx}+\gamma\psi_{xx}-du_{x}-\alpha_{3}%
		\varphi-\alpha_{2}\psi-\tau_{3}\varphi_{t}-\tau_{4}\psi_{t} &
		\mbox{in}~(0,\pi)\times\mathbb{R}_+,
	\end{array}
	\right.  \label{A}%
\end{equation}
where $u$ is the transversal displacement of a one-dimensional porous elastic
solid of length $\pi$, $\varphi$ and $\psi$ are the porous unknown
functions one associated to the pores in the skeleton and the other associated
with the fissures in the material body. The parameter $\rho,\kappa_{1}$ and
$\kappa_{2},$ which assumed to be strictly positive, are the mass density, and
the products of the mass density by the equilibrated inertia, respectively.
The coefficients $\mu,\alpha,\beta,\gamma,\alpha_{1},\alpha_{2},\alpha
_{3},b,d,\tau_{1}$,\\$\tau_{2},\tau_{3}$ and $\tau_{4}$ are parameters 
related on the properties of the material. We assume that they satisfy some
restrictions that will be specified later.

The system considered here, represented an elastic solid with double porosity
structure in the framework of the theory of elastic materials with voids
developed by Nunziato-Cowin \cite{Cowin}. This approach has been used by
Ie\c{s}an and Quintanilla \cite{Quintanilla Iesan} to derive a new theory of
thermoelastic solids which have a double porosity structure. In contrast to
the classical theory the new one is not based on Darcy's law, and the porosity
structure in the case of equilibrium is influenced by the displacement field.

The origin of the classical theory of elastic materials with
double porosity goes back to the works of Barenblatt \emph{et al.}
\cite{Barenblatt1,Barenblatt2}. The authors introduced two liquid pressures at
each point of the material which allows the body to have a double porosity
structure: a macro porosity connected to pores in the body and a micro
porosity connected to fissures in the skeleton.

In the last few years, a great interest has been given to the analysis of the
longtime behavior of solutions of porous thermoelastic problems. A part of
this interest stems from the need to have general results that explain the
experimental observations of engineers. The earliest contribution in this
direction was achieved by Quintanilla \cite{Quintanilla}. He considered the
porous elastic system
\begin{equation}\label{Q}
	\left\{
	\begin{array}
		[c]{lc}%
		\rho_{0}u_{tt}=\mu u_{xx}+\beta\varphi_{x} & \mbox{in}~~(0,\pi)\times
		(0,+\infty),\\
		\rho_{0}\kappa\varphi_{tt}=\alpha\varphi_{xx}-\beta u_{x}-\xi\varphi
		-\tau\varphi_{t} & \mbox{in}~~(0,\pi)\times(0,+\infty),
	\end{array}
	\right.
\end{equation}
where $u$ is the transversal displacement and $\varphi$ is the volume
fraction. He used Hurwitz theorem and showed that the porous dissipation
$\tau\varphi_{t}$ is not powerful enough to produce an exponential stability.

Several dissipative mechanisms have been examined to stabilize system (\ref{Q}) exponentially. Casas and Quintanilla \cite{Casas2} coupled system (\ref{A}) (for $\tau=0$)
with the heat equation, and proved the non-exponential stability. However, if thermal and porous dissipations or micro-thermal and viscoelastic dissipations are combined then the solution decays exponentially \cite{Casas1,Magana Quintanilla}.

Apalara \cite{Apalara} considered the porous thermoelastic system%
\[
\left\{
\begin{array}
	[c]{ll}%
	\rho u_{tt}-\mu u_{xx}-b\phi_{x}=0, & in\;\left(  0,1\right)  \times\left(
	0,+\infty\right) \\
	J\phi_{tt}-\delta\phi_{xx}+bu_{x}+\xi\phi+\tau\phi_{t}=0, & in\;\left(
	0,1\right)  \times\left(  0,+\infty\right)  ,
\end{array}
\right.
\]
with different boundary conditions. He investigated the case of equal wave speeds $\dfrac{\mu}{\rho}=\dfrac{\delta}{J}$ and proved that the unique
dissipation in the porous equation leads to an exponential stability. He also replaced the frictional damping $\tau\phi_t$ by the memory term $\displaystyle\int_0^tg(t-s)\phi_{xx} (s)ds$ and obtained a general rate of decay \cite{Apalara1}. We notice that the results of \cite{Apalara,Apalara1} disprove Maga\~{n}a and Quintanilla's  claim that a
porous-elastic system with a single dissipation mechanism can not be exponentially stable \cite{Magana Quintanilla}.

In the context of double porous thermoelasticity, Bazarra \emph{et al.}
\cite{Bazarra} considered the system%
\begin{equation}
	\left\{
	\begin{array}
		[c]{l}%
		\rho u_{tt}=\mu u_{xx}+b\varphi_{x}+d\psi_{x}-\beta\theta_{x},\\
		\kappa_{1}\varphi_{tt}=\alpha\varphi_{xx}+b_{1}\psi_{xx}-bu_{x}-\alpha
		_{1}\varphi-\alpha_{3}\psi+\gamma_{1}\theta-\varepsilon_{1}\varphi
		_{t}-\varepsilon_{2}\psi_{t},\\
		\kappa_{2}\psi_{tt}=b_{1}\varphi_{xx}+\gamma\psi_{xx}-du_{x}-\alpha_{3}%
		\varphi-\alpha_{2}\psi+\gamma_{2}\theta-\varepsilon_{3}\varphi_{t}%
		-\varepsilon_{4}\psi_{t},\\
		c\theta_{t}=\kappa\theta_{xx}-\beta u_{tx}-\gamma_{1}\varphi_{t}-\gamma
		_{2}\psi_{t},
	\end{array}
	\right.  \label{B}%
\end{equation}
with the boundary conditions%
\[
u\left(  x,t\right)  =\varphi_{x}\left(  x,t\right)  =\psi_{x}\left(
x,t\right)  =\theta_{x}\left(  x,t\right)  =0,\;x=0,\;x=\pi,\;\;\forall t\geq
0.
\]
They proved that the solution decays exponentially when porous dissipation is
assumed for each porous equations. If the dissipation is considered only on
one porous structure, the solution cannot be asymptotically stable in
general. However, they give a sufficient conditions for which the solutions
decay exponentially. See also \cite{Bazarra1}.

Recently, Nemsi and Fareh \cite{Nemsi} proved that the solution of the system
\begin{equation}
	\left\{
	\begin{array}
		[c]{lc}%
		\rho u_{tt}=\mu u_{xx}+b\varphi_{x}+d\psi_{x}+\lambda u_{txx}, &
		\mbox{in}~~(0,L)\times(0,\infty
		),\\
		\kappa_{1}\varphi_{tt}=\alpha\varphi_{xx}+b_{1}\psi_{xx}-bu_{x}-\alpha
		_{1}\varphi-\alpha_{3}\psi-\tau_{1}\varphi_{t} & \mbox{in}~~(0,L)\times(0,\infty
		),\\
		\kappa_{2}\psi_{tt}=b_{1}\varphi_{xx}+\gamma\psi_{xx}-du_{x}-\alpha_{3}%
		\varphi-\alpha_{2}\psi-\tau_{2}\psi_{t} & \mbox{in}~~(0,L)\times(0,\infty
		),
	\end{array}
	\right.  \label{C}%
\end{equation}
with the boundary conditions%
\[%
\begin{array}
	[c]{ll}%
	u(t,0)=u(t,L)=\varphi_{x}(t,0)=\varphi_{x}(t,L)=\psi_{x}(t,0)=\psi
	_{x}(t,L)=0 & \mbox{in}~(0,\infty),
\end{array}
\]
decays exponentially without any assumption on the wave speeds.

In this paper we consider system (\ref{A}) subjected to the initial data%
\begin{align}
	u\left(  x,0\right)   &  =u_{0}\left(  x\right),~u_{t}\left(  x,0\right)
	=u_{1}\left(  x\right),\nonumber\label{A1}\\
	\varphi\left(  x,0\right)   &  =\varphi_{0}\left(  x\right),~\varphi
	_{t}\left(  x,0\right)  =\varphi_{1}\left(  x\right),\\
	\psi\left(  x,0\right)   &  =\psi_{0}\left(  x\right),~\psi_{t}\left(
	x,0\right)  =\psi_{1}\left(  x\right) \nonumber
\end{align}
for all $x\in\left(  0,\pi\right)  $ and the  boundary conditions%
\begin{equation}
	u_{x}\left(  0,t\right)  =u_{x}\left(  \pi,t\right)  =\varphi\left(
	0,t\right)  =\varphi\left(  \pi,t\right)  =\psi\left(  0,t\right)
	=\psi\left(  \pi,t\right)  =0,\;t\geq0, \label{A2}%
\end{equation}
or%
\begin{equation}
	u\left(  0,t\right)  =u\left(  \pi,t\right)  =\varphi_{x}\left(  0,t\right)
	=\varphi_{x}\left(  \pi,t\right)  =\psi_{x}\left(  0,t\right)  =\psi
	_{x}\left(  \pi,t\right)  =0,\;t\geq0. \label{A3}%
\end{equation}

Note that system (\ref{A}) coincides with  (\ref{B}) in the isothermal
case ($\beta=\gamma_{1}=\gamma_{2}=0$) and with  (\ref{C}) for
$\lambda=0.$ Therefore, system (\ref{A}) lacks thermal and
viscoelastic dissipations. Moreover, in somehow the two porous functions can be viewed as a multivalued function $(\varphi,\psi)$, consequently, system (\ref{A}) can be viewed as a porous elastic system with only one dissipation. Thus, our exponential stability extends the result of \cite{Apalara,Apalara1} and our stability number generalizes the wave speeds equality. 

We assume that the constitutive coefficients
$\mu,\alpha,\beta,\gamma,\alpha_{1}$ and $\alpha_{2}$ are positive and as
coupling is considered the coefficients $b,d$ must not be zero simultaneously.
Next, we define the energy associated with the solution $\left(
u,\varphi,\psi\right)  $ of system (\ref{A}) by%
\begin{align}
	E\left(  t\right)   &  :=\frac{1}{2}\int_{0}^{\pi}\left[  \rho u_{t}%
	^{2}+\kappa_{1}\varphi_{t}^{2}+\kappa_{2}\psi_{t}^{2}+\mu u_{x}^{2}%
	+\alpha\varphi_{x}^{2}+\gamma\psi_{x}^{2}+\alpha_{1}\varphi^{2}+\alpha_{2}%
	\psi^{2}\right. \label{E}\\
	&  \left.  +2\beta\varphi_{x}\psi_{x}+2bu_{x}\varphi+2du_{x}\psi+2\alpha
	_{3}\varphi\psi\right]  dx.\nonumber
\end{align}

\begin{remark}
	\label{Rk1}To guarantee that the energy $E\left(  t\right)  $ is a positive
	definite form, we assume that the matrix%
	\[
	A=\left(
	\begin{array}
		[c]{ccccc}%
		\mu & b & d & 0 & 0\\
		b & \alpha_{1} & \alpha_{3} & 0 & 0\\
		d & \alpha_{3} & \alpha_{2} & 0 & 0\\
		0 & 0 & 0 & \alpha & \beta\\
		0 & 0 & 0 & \beta & \gamma
	\end{array}
	\right)
	\]
	is positive definite.
	
	Indeed, since any principal submatrix of a positive definite matrix is also
	positive definite, then
	\begin{equation}
		\left(  \alpha_{1}-\frac{b^{2}}{\mu}\right)  \left(  \alpha_{2}-\frac{d^{2}%
		}{\mu}\right)  -\left(  \alpha_{3}-\frac{bd}{\mu}\right)  ^{2}>0, \label{C1}%
	\end{equation}%
	\begin{equation}
		\alpha_{1}\mu-b^{2}>0,\;\;\alpha_{2}\mu-d^{2}>0,\,\alpha_{1}\alpha_{2}%
		-\alpha_{3}^{2}>0\;\mbox{and}\;\alpha\gamma-\beta^{2}>0. \label{C2}%
	\end{equation}
	Therefore,%
	\[
	\alpha\varphi_{x}^{2}+\gamma\psi_{x}^{2}+2\beta\varphi_{x}\psi_{x}=\frac{1}%
	{2}\left(  \alpha-\frac{\beta^{2}}{\gamma}\right)  \varphi_{x}^{2}+\frac{1}%
	{2}\left(  \gamma-\frac{\beta^{2}}{\alpha}\right)  \psi_{x}^{2}%
	\]%
	\[
	+\frac{\alpha}{2}\left(  \varphi_{x}+\frac{\beta}{\alpha}\psi_{x}\right)
	^{2}+\frac{\gamma}{2}\left(  \psi_{x}+\frac{\beta}{\gamma}\varphi_{x}\right)
	^{2}\geq0.
	\]
	Moreover, there exists a $\varepsilon>0$ such that the matrix
	\[
	B=\left(
	\begin{array}
		[c]{ccc}%
		\mu-\varepsilon & b & d\\
		b & \alpha_{1}-\varepsilon & \alpha_{3}\\
		d & \alpha_{3} & \alpha_{2}-\varepsilon
	\end{array}
	\right)
	\]
	still positive definite. Thus,%
	\[
	E\left(  t\right)  =\frac{1}{2}\int_{0}^{\pi}\left(  \rho u_{t}^{2}+\kappa
	_{1}\varphi_{t}^{2}+\kappa_{2}\psi_{t}^{2}+\varepsilon u_{x}^{2}%
	+\varepsilon\varphi^{2}+\varepsilon\psi^{2}\right)  dx
	\]%
	\[
	+\frac{1}{2}\int_{0}^{\pi}\left[  \left(  \mu-\varepsilon\right)  u_{x}%
	^{2}+2bu_{x}\varphi+2du_{x}\psi+\left(  \alpha_{1}-\varepsilon\right)
	\varphi^{2}+2\alpha_{3}\varphi\psi+\left(  \alpha_{2}-\varepsilon\right)
	\psi^{2}\right]  dx
	\]%
	\[
	+\frac{1}{2}\int_{0}^{\pi}\left[  \alpha\left(  \varphi_{x}+\frac{\beta
	}{\alpha}\psi_{x}\right)  ^{2}+\gamma\left(  \psi_{x}+\frac{\beta}{\gamma
	}\varphi_{x}\right)  ^{2}\right]  dx
	\]%
	\[
	+\frac{1}{2}\left(  \alpha-\frac{\beta^{2}}{\gamma}\right)  \int_{0}^{\pi
	}\left\vert \varphi_{x}\right\vert ^{2}dx+\frac{1}{2}\left(  \gamma
	-\frac{\beta^{2}}{\alpha}\right)  \int_{0}^{\pi}\left\vert \psi_{x}\right\vert
	^{2}dx\geq0.
	\]
	
\end{remark}

The rest of the paper is organized as follows: in Section 2, we prove the well
posedness of the problem determined by (\ref{A}), (\ref{A1}) and (\ref{A3}).
In Section 3 we define two stability numbers $\chi_{0}$ and $\chi_{1}$ and
prove, by the use of the multiplier method, that the solution decays
exponentially, provided that $\chi_{0}=0$ and $\chi_{1}\neq0$. Section 4 is
devoted to the proof of the lack of the exponential decay when $\chi_{1}=0$ or
$\chi_{0}\neq0$.

\section{Existence and uniqueness}

In this section we prove the existence and the uniqueness of a solution to the
problem determined by system (\ref{A}) and conditions (\ref{A1}) and
(\ref{A3}), the case of boundary conditions (\ref{A2}) is similar.

As Neumann boundary conditions are considered for $\varphi$ and $\psi,$
Poincar\'{e}'s inequality cannot be applied. From the second and the third
equations of (\ref{A}) and boundary conditions (\ref{A3}), we have%
\begin{equation}%
	\begin{array}
		[c]{l}%
		\displaystyle\frac{d^{2}}{dt^{2}}\int_{0}^{\pi}\varphi dx=-\frac{\alpha_{1}%
		}{\kappa_{1}}\int_{0}^{\pi}\varphi dx-\frac{\alpha_{3}}{\kappa_{1}}\int
		_{0}^{\pi}\psi dx-\frac{\tau_{1}}{\kappa_{1}}\frac{d}{dt}\int_{0}^{\pi}\varphi
		dx-\frac{\tau_{2}}{\kappa_{1}}\frac{d}{dt}\int_{0}^{\pi}\psi dx,\\
		\\
		\displaystyle\frac{d^{2}}{dt^{2}}\int_{0}^{\pi}\psi dx=-\frac{\alpha_{3}%
		}{\kappa_{2}}\int_{0}^{\pi}\varphi dx-\frac{\alpha_{2}}{\kappa_{2}}\int
		_{0}^{\pi}\psi dx-\frac{\tau_{3}}{\kappa_{2}}\frac{d}{dt}\int_{0}^{\pi}%
		\varphi_{t}dx-\frac{\tau_{4}}{\kappa_{2}}\frac{d}{dt}\int_{0}^{\pi}\psi_{t}dx.
	\end{array}
	\label{X}%
\end{equation}
So if we set $X=\left(  \int_{0}^{\pi}\varphi dx,\int_{0}^{\pi}\varphi
_{t}dx,\int_{0}^{\pi}\psi dx,\int_{0}^{\pi}\psi_{t}dx\right)  ^{T}$ then
(\ref{X}) can be written%
\begin{equation}
	X_{t}\left(  t\right)  =MX\left(  t\right)  ,\ X\left(  0\right)  =X_{0},
	\label{X1}%
\end{equation}
where%
\[
M=\left(
\begin{array}
	[c]{cccc}%
	0 & 1 & 0 & 0\\
	-\frac{\alpha_{1}}{\kappa_{1}} & -\frac{\tau_{1}}{\kappa_{1}} & -\frac
	{\alpha_{3}}{\kappa_{1}} & -\frac{\tau_{2}}{\kappa_{1}}\\
	0 & 0 & 0 & 1\\
	-\frac{\alpha_{3}}{\kappa_{2}} & -\frac{\tau_{3}}{\kappa_{2}} & -\frac
	{\alpha_{2}}{\kappa_{2}} & -\frac{\tau_{4}}{\kappa_{2}}%
\end{array}
\right)
\]
and%
\[
X_{0}=\left(  \int_{0}^{\pi}\varphi_{0}dx,\int_{0}^{\pi}\varphi_{1}dx,\int
_{0}^{\pi}\psi_{0}dx,\int_{0}^{\pi}\psi_{1}dx\right)  ^{T}.
\]
Solving (\ref{X1}) we get%
\[
X\left(  t\right)  =\exp\left(  tM\right)  X_{0},
\]
in particular,%
\[%
\begin{array}
	[c]{cc}%
	\displaystyle\int_{0}^{\pi}\varphi dx=\sum_{k=1}^{4}\left(  \exp\left(
	tM\right)  \right)  _{1k}X_{0k}, & \displaystyle\int_{0}^{\pi}\psi
	dx=\sum_{j=1}^{4}\left(  \exp\left(  tM\right)  \right)  _{3k}X_{0k}.
\end{array}
\]
Therefore, if we set%
\[%
\begin{array}
	[c]{cc}%
	\displaystyle\bar{\varphi}=\varphi-\sum_{k=1}^{4}\left(  \exp\left(
	tM\right)  \right)  _{1k}X_{0k}, & \displaystyle\bar{\psi}=\psi-\sum_{j=1}%
	^{4}\left(  \exp\left(  tM\right)  \right)  _{3k}X_{0k},
\end{array}
\]
then $\left(  u,\overline{\varphi},\overline{\psi}\right)  $ solves (\ref{A})
with boundary conditions (\ref{A3}), and we have
\[
\int_{0}^{\pi}\overline{\varphi}dx=\int_{0}^{\pi}\overline{\psi}dx=0,
\]
which allows to apply Poincar\'{e}'s inequality. In the sequel we will work
with $\bar{\varphi}$ and $\bar{\psi}$ but for convenience, we write
$\varphi,\psi$ instead of $\bar{\varphi},\bar{\psi}$ respectively.

Furthermore, as porous dissipations are considered, the weights of porous
dampings $\tau_{1},\tau_{2},\tau_{3}$ and $\tau_{4}$ are assumed to satisfy%
\begin{equation}
	\tau_{1}>0,\;4\tau_{1}\tau_{4}>\left(  \tau_{2}+\tau_{3}\right)  ^{2}.
	\label{V}%
\end{equation}

\begin{lemma}
	The energy $E\left(  t\right)  $ satisfies along the solution $\left(
	u,\varphi,\psi\right)  $ of (\ref{A})-(\ref{A2}) the estimate%
	\begin{equation}
		E^{\prime}\left(  t\right)  =-\tau_{1}\int_{0}^{\pi}\varphi_{t}^{2}dx-\tau
		_{4}\int_{0}^{\pi}\psi_{t}^{2}dx-\left(  \tau_{2}+\tau_{3}\right)  \int
		_{0}^{\pi}\varphi_{t}\psi dx \label{B1}%
	\end{equation}
	and we have%
	\[
	E^{\prime}\left(  t\right)  =-\frac{1}{2}\left(  \tau_{1}-\frac{\left(
		\tau_{2}+\tau_{3}\right)  ^{2}}{4\tau_{2}}\right)  \int_{0}^{\pi}\varphi
	_{t}^{2}dx-\frac{1}{2}\left(  \tau_{2}-\frac{\left(  \tau_{2}+\tau_{3}\right)
		^{2}}{4\tau_{1}}\right)  \int_{0}^{\pi}\psi_{t}^{2}dx
	\]%
	\begin{equation}
		-\frac{\tau_{1}}{2}\int_{0}^{\pi}\left(  \varphi_{t}+\frac{\left(  \tau
			_{2}+\tau_{3}\right)  }{2\tau_{1}}\psi_{t}\right)  ^{2}dx-\frac{\tau_{2}}%
		{2}\int_{0}^{\pi}\left(  \psi_{t}+\frac{\left(  \tau_{2}+\tau_{3}\right)
		}{2\tau_{2}}\varphi_{t}\right)  ^{2}dx\leq0. \label{B2}%
	\end{equation}
	
\end{lemma}

\begin{proof}
	Multiplying the equations of (\ref{A}) by $u_{t},\varphi_{t}$ and $\psi_{t}$
	respectively, then integrating with respect to $x$ over $\left(  0,\pi\right)
	$ and using integration by parts and boundary conditions (\ref{A3}), the
	estimate (\ref{B1}) follows immediately.
\end{proof}

To prove the well-posedness, we use a semigroup approach. First, we introduce
the energy space%
\[
\mathcal{H}=H_{0}^{1}\left(  0,\pi\right)  \times L^{2}\left(  0,\pi\right)
\times H_{\ast}^{1}\left(  0,\pi\right)  \times L_{\ast}^{2}\left(
0,\pi\right)  \times H_{\ast}^{1}\left(  0,\pi\right)  \times L_{\ast}%
^{2}\left(  0,\pi\right)  ,
\]
where,%
\begin{align*}
	H_{\ast}^{1}\left(  0,\pi\right)   &  =\left\{  \phi\in H^{1}\left(
	0,\pi\right)  :\int_{0}^{\pi}\phi\left(  x\right)  dx=0\right\}  ,\,\\
	L_{\ast}^{2}\left(  0,\pi\right)   &  =\left\{  \phi\in L^{2}\left(
	0,\pi\right)  :\int_{0}^{\pi}\phi\left(  x\right)  dx=0\right\}  .
\end{align*}
We note that $L_{\ast}^{2}\left(  0,\pi\right)  $ and $H_{\ast}^{1}\left(
0,\pi\right)  $ are closed subspaces of $L^{2}\left(  0,\pi\right)  $ and
$H^{1}\left(  0,\pi\right)  $ respectively. Thus, they are Hilbert spaces and
so $\mathcal{H}$ is.

Next, we rewrite the system (\ref{A}) in the setting of Lumer-Phillips
theorem, to do so we introduce the new variables $v=u_{t},\phi=\varphi_{t}$
and $w=\psi_{t}$ the system (\ref{A}) becomes
\[
\left\{
\begin{array}
	[c]{l}%
	u_{t}=v\\
	v_{t}=\dfrac{1}{\rho}\left(  \mu u_{xx}+b\varphi_{x}+d\psi_{x}\right)  ,\\
	\varphi_{t}=\phi\\
	\phi_{t}=\dfrac{1}{\kappa_{1}}\left(  \alpha\varphi_{xx}+\beta\psi_{xx}%
	-bu_{x}-\alpha_{1}\varphi-\alpha_{3}\psi-\tau_{1}\phi-\tau_{2}w\right)  ,\\
	\psi_{t}=w\\
	w_{t}=\dfrac{1}{\kappa_{2}}\left(  \beta\varphi_{xx}+\gamma\psi_{xx}%
	-du_{x}-\alpha_{3}\varphi-\alpha_{2}\psi-\tau_{3}\phi-\tau_{4}w\right)  ,
\end{array}
\right.
\]
which can be written%
\begin{equation}
	\left\{
	\begin{array}
		[c]{l}%
		U_{t}=\mathcal{A}U,\\
		U\left(  0\right)  =U_{0},
	\end{array}
	\right.  \label{9}%
\end{equation}
where $\mathcal{A}:D\left(  \mathcal{A}\right)  \subset\mathcal{H}%
\longrightarrow\mathcal{H}$ is the operator defined by
\[
\mathcal{A=}\left(
\begin{array}
	[c]{cccccc}%
	0 & I & 0 & 0 & 0 & 0\\
	\frac{\mu}{\rho}\partial_{xx} & 0 & \frac{b}{\rho}\partial_{x} & 0 & \frac
	{d}{\rho}\partial_{x} & 0\\
	0 & 0 & 0 & I & 0 & 0\\
	-\frac{b}{\kappa_{1}}\partial_{x} & 0 & \frac{\alpha}{\kappa_{1}}\partial
	_{xx}-\frac{\alpha}{\kappa_{1}} & -\frac{\tau_{1}}{\kappa_{1}} & \frac{\beta
	}{\kappa_{1}}\partial_{xx}-\frac{\alpha_{3}}{\kappa_{1}} & -\frac{\tau_{2}%
	}{\kappa_{1}}\\
	0 & 0 & 0 & 0 & 0 & I\\
	-\frac{d}{\kappa_{2}}\partial_{x} & 0 & \frac{\beta}{\kappa_{2}}\partial
	_{xx}-\alpha_{3} & -\frac{\tau_{3}}{\kappa_{2}} & \frac{\gamma}{\kappa_{2}%
	}\partial_{xx}-\frac{\alpha_{2}}{\kappa_{2}} & -\frac{\tau_{4}}{\kappa_{2}}%
\end{array}
\right)
\]
with domain%
\[
D\left(  \mathcal{A}\right)  =\left(  H^{2}\left(  0,\pi\right)  \cap
H_{0}^{1}\left(  0,\pi\right)  \right)  \times H_{0}^{1}\left(  0,\pi\right)
\times H_{\ast}^{2}\left(  0,\pi\right)  \times H_{\ast}^{1}\left(
0,\pi\right)  \times H_{\ast}^{2}\left(  0,\pi\right)  \times H_{\ast}%
^{1}\left(  0,\pi\right)  .
\]
Here $I$ is the identity operator, $\partial$ denotes the derivative with
respect to $x$ and
\[
H_{\ast}^{2}\left(  0,\pi\right)  =\left\{  \phi\in H^{2}\left(  0,\pi\right)
:\phi\left(  0\right)  =\phi\left(  \pi\right)  =0\right\}  .
\]
The following two theorems are useful to proof our well posedness result.

\begin{theorem}
	{\normalfont(Lumer-Phillips)} \cite{Pazy,Vrabie} Let $\mathcal{A}:
	D(\mathcal{A})\subset H\longrightarrow H$ be a densely defined operator. Then
	$\mathcal{A}$ generates a C$_{0}$-semigroup of contractions on $H$ if and only if
	
	\begin{itemize}
		\item[\textit{(i)}] $\mathcal{A}$ is dissipative;
		
		\item[\textit{(ii)}] there exists $\lambda>0$ such that $\lambda
		I-\mathcal{A}$ is surjective.
	\end{itemize}
\end{theorem}

\begin{theorem}
	\cite{Vrabie} Let $\mathcal{A}: D(\mathcal{A})\subset H\longrightarrow H$ be
	the infinitesimal generator of a C$_{0}$-semigroup $\{S(t);t\geq0\}$. Then, for
	each $\xi\in D(\mathcal{A})$ and each $t\geq0$, we have $S(t)\xi\in
	D(\mathcal{A})$, and the mapping
	\begin{align*}
		&  t\longrightarrow S(t)\xi
	\end{align*}
	is of class $C^{1}$ on $\left[  0,+\infty\right)  $ and satisfies
	\[
	\frac{d}{dt}(S(t)\xi)=\mathcal{A}S(t)\xi=S(t)\mathcal{A}\xi.
	\]
	
\end{theorem}

Now, we state and prove the well-posedness theorem of the problem (\ref{A}),
(\ref{A1}) and (\ref{A3}).

\begin{theorem}
	\label{TH1}For any $U_{0}=\left(  u_{0},u_{1},\varphi_{0},\varphi_{1},\psi
	_{0},\psi_{1}\right)  \in\mathcal{H}$, the problem (\ref{A}),(\ref{A1}) and
	(\ref{A3}) has a unique weak solution $\left(  u,\varphi,\psi\right)  $
	satisfies the property:%
	\begin{align*}
		u &  \in C\left(  \left[  0,+\infty\right[  ;H_{0}^{1}\left(  0,\pi\right)
		\right)  \cap C^{1}\left(  \left[  0,+\infty\right[  ;L^{2}\left(
		0,\pi\right)  \right)  ,\\
		\varphi,\psi &  \in C\left(  \left[  0,+\infty\right[  ;H_{\ast}^{1}\left(
		0,\pi\right)  \right)  \cap C^{1}\left(  \left[  0,+\infty\right[  ;L_{\ast
		}^{2}\left(  0,\pi\right)  \right)  .
	\end{align*}
	Moreover, if $U_{0}\in D\left(  \mathcal{A}\right)  ,$ the solution $\left(
	u,\varphi,\psi\right)  $ satisfies%
	\begin{align*}
		u &  \in C\left(  \left[  0,+\infty\right[  ;H^{2}\cap H_{0}^{1}\left(
		0,\pi\right)  \right)  \cap C^{1}\left(  \left[  0,+\infty\right[  ;H_{0}%
		^{1}\left(  0,\pi\right)  \right)  \cap C^{2}\left(  \left[  0,+\infty\right[
		;L^{2}\left(  0,\pi\right)  \right)  ,\\
		\varphi,\psi &  \in C\left(  \left[  0,+\infty\right[  ;H_{\ast}^{2}\left(
		0,\pi\right)  \right)  \cap C^{1}\left(  \left[  0,+\infty\right[  ;H_{\ast
		}^{1}\left(  0,\pi\right)  \right)  \cap C^{2}\left(  \left[  0,+\infty
		\right[  ;L_{\ast}^{2}\left(  0,\pi\right)  \right)  .
	\end{align*}
\end{theorem}	
	
	\begin{proof}
		According to the Lumer-Phillips theorem, it suffices to prove that the
		operator $\mathcal{A}$ is dissipative and maximal.
		
		First, we have for any $U\in D\left(  \mathcal{A}\right)  $,%
		\[
		\operatorname{Re}\left\langle \mathcal{A}U,U\right\rangle _{\mathcal{H}}%
		=-\tau_{1}\int_{0}^{\pi}\varphi_{t}^{2}dx-\left(  \tau_{2}+\tau_{3}\right)
		\int_{0}^{\pi}\varphi_{t}\psi_{t}dx-\tau_{4}\int_{0}^{\pi}\psi_{t}^{2}%
		dx\leq0.
		\]
		Therefore, $\mathcal{A}$ is dissipative.
		
		Secondly, \bigskip let $F=\left(  f_{1},f_{2},f_{3},f_{4},f_{5},f_{6}\right)
		\in\mathcal{H},$ and find $U\in D\left(  \mathcal{A}\right)  $ such that
		$\mathcal{A}U=\mathcal{F}$, that is,%
		\[
		\left\{
		\begin{array}
			[c]{l}%
			v=f_{1}\in H_{0}^{1}\\
			\mu u_{xx}+b\varphi_{x}+d\psi_{x}=\rho f_{2}\in L^{2},\\
			\phi=f_{3}\in H_{\ast}^{1}\\
			\alpha\varphi_{xx}+\beta\psi_{xx}-bu_{x}-\alpha_{1}\varphi-\alpha_{3}\psi
			-\tau_{1}\phi-\tau_{2}w=\kappa_{1}f_{4}\in L_{\ast}^{2},\\
			w=f_{5}\in H_{\ast}^{1}\\
			\beta\varphi_{xx}+\gamma\psi_{xx}-du_{x}-\alpha_{3}\varphi-\alpha_{2}\psi
			-\tau_{3}\phi-\tau_{4}w=\kappa_{2}f_{6}\in L_{\ast}^{2}.
		\end{array}
		\right.
		\]
		From the first, the third and the fifth equations we have $v\in H_{0}%
		^{1}\left(  0,\pi\right)  $ and $\phi,w\in H_{\ast}^{1}\left(  0,\pi\right)
		.$
	Substituting $v,\phi$ and $w$ by $f_{1},f_{3}$ and $f_{5}$ respectively we
	obtain
		\begin{equation}
			\left\{
			\begin{array}
				[c]{l}%
				\displaystyle\mu u_{xx}+b\varphi_{x}+d\psi_{x}=\rho f_{2}\in L^{2},\\
				\displaystyle\alpha\varphi_{xx}+\beta\psi_{xx}-bu_{x}-\alpha_{1}\varphi
				-\alpha_{3}\psi=\kappa_{1}f_{4}+\tau_{1}f_{3}+\tau_{2}f_{5}=g_{1}\in L_{\ast
				}^{2},\\
				\displaystyle\beta\varphi_{xx}+\gamma\psi_{xx}-du_{x}-\alpha_{3}\varphi
				-\alpha_{2}\psi=\kappa_{2}f_{6}+\tau_{3}f_{3}+\tau_{4}f_{5}=g_{2}\in L_{\ast
				}^{2},
			\end{array}
			\right.  \label{b}%
		\end{equation}
		Taking the $L^{2}$-product of (\ref{b})$_{1}$ ,(\ref{b})$_{2}$ and
		(\ref{b})$_{3}$ by $u^{\ast},\varphi^{\ast}$ and $\psi^{\ast}$ respectively,
		using integration by parts and adding the obtained equations, we arrive at%
	
	\begin{equation}
		a\left(  V,V^{\ast}\right)  =L\left(  V^{\ast}\right)  ,\label{d}%
	\end{equation}
	where, $a$ is the bilinear form defined on $\mathcal{W}=\left(  H_{0}%
	^{1}\left(  0,\pi\right)  \times H_{\ast}^{1}\left(  0,\pi\right)  \times
	H_{\ast}^{1}\left(  0,\pi\right)  \right)  $ and for $V=\left(u,\varphi,\psi\right),~\left(u^{\ast},\varphi^{\ast},\psi^{\ast}\right)\in\mathcal{W}$, by%
	\begin{align*}
	a\left(  V,V^{\ast}\right)  =&\mu\int_{0}^{\pi}u_{x}u_{x}^{\ast}dx+b\int
	_{0}^{\pi}\varphi u_{x}^{\ast}dx+d\int_{0}^{\pi}\psi u_{x}^{\ast}dx+\alpha
	\int_{0}^{\pi}\varphi_{x}\varphi_{x}^{\ast}dx
	\\
	&+\beta\int_{0}^{\pi}\psi_{x}\varphi_{x}^{\ast}dx+b\int_{0}^{\pi}u_{x}%
	\varphi^{\ast}dx+\alpha_{1}\int_{0}^{\pi}\varphi\varphi^{\ast}dx+\alpha
	_{3}\int_{0}^{\pi}\psi\varphi^{\ast}dx
	\\
	&+\beta\int_{0}^{\pi}\varphi_{x}\psi_{x}^{\ast}dx+\gamma\int_{0}^{\pi}\psi
	_{x}\psi_{x}^{\ast}dx+d\int_{0}^{\pi}u_{x}\psi^{\ast}dx
	\\
	&+\alpha_{3}\int_{0}^{\pi}\varphi\psi^{\ast}dx+\alpha_{2}\int_{0}^{\pi}\psi
	\psi^{\ast}dx
	\end{align*}
	and $L$ is the linear form defined by
	\[
	L\left(  V^{\ast}\right)  =-\rho\int_{0}^{\pi}f_{2}u^{\ast}dx-\int_{0}^{\pi
	}g_{1}\varphi^{\ast}dx-\int_{0}^{\pi}g_{2}\psi^{\ast}dx.
	\]
	Clearly, $a$ and $L$ are continuous. Furthermore, from Remark \ref{Rk1}, there
	exists $\varepsilon>0,$ such that%
	\begin{align*}
	a\left(  V,V\right)  =&\mu\int_{0}^{\pi}u_{x}^{2}dx+\alpha\int_{0}^{\pi}%
	\varphi_{x}^{2}dx+\gamma\int_{0}^{\pi}\psi_{x}^{2}dx+\alpha_{1}\int_{0}^{\pi
	}\varphi^{2}dx+\alpha_{2}\int_{0}^{\pi}\psi^{2}dx
	\\
	&+2b\int_{0}^{\pi}u_{x}\varphi dx+2d\int_{0}^{\pi}\psi u_{x}dx+2\beta\int
	_{0}^{\pi}\varphi_{x}\psi_{x}dx+2\alpha_{3}\int_{0}^{\pi}\psi\varphi dx,
	\\
	\geq&\frac{1}{2}\left(  \alpha-\frac{\beta^{2}}{\gamma}\right)  \int_{0}^{\pi}\varphi_{x}%
	^{2}dx+\frac{1}{2}\left(  \gamma-\frac{\beta^{2}}{\alpha}\right)  \int_{0}^{\pi}\psi_{x}
	^{2}dx+\varepsilon\int_{0}^{\pi}\left(  u_{x}^{2}+\varphi^{2}+\psi^{2}\right)dx.
	\end{align*}
	Thus,
	\[
	a\left(  V,V\right)  \geq c\left\Vert V\right\Vert _{\mathcal{W}}^{2},%
	\]
	for $c=\dfrac{1}{2}\min\{\displaystyle\alpha-\frac{\beta^{2}}{\gamma
	},\displaystyle\gamma-\frac{\beta^{2}}{\alpha},2\varepsilon\},$ which shows
	that $a$ is coercive. Therefore, Lax-Milgram theorem ensures the existence of
	a unique $V=\left(  u,\varphi,\psi\right)  $\bigskip$\in\mathcal{W}$ satisfying%
	\[%
	\begin{array}
		[c]{cc}%
		a\left(  V,V^{\ast}\right)  =L\left(  V^{\ast}\right)  , & \forall V^{\ast}%
		\in\mathcal{W}.
	\end{array}
	\]
	Now, taking $V^{\ast}=(u^{\ast},0,0)$ in (\ref{d}) we get%
	\begin{equation}
		\mu\int_{0}^{\pi}u_{x}u_{x}^{\ast}=-\int_{0}^{\pi}\left(  \rho f_{2}%
		-b\varphi_{x}-d\psi_{x}\right)  u^{\ast}dx,\,\forall u^{\ast}\in H_{0}^{1}.
	\end{equation}
	The elliptic regularity theory shows that%
	\[
	u\in H^{2}\left(  0,\pi\right),
	\]
	with
	\[u_{xx}=\frac{1}{\mu}\left(  \rho f_{2}%
	-b\varphi_{x}-d\psi_{x}\right),\]
	which solves the first equation of (\ref{b}).\\ 
	Next, let $\varphi^{\ast}\in H_{0}^{1}\left(  0,\pi\right)  $ and define
	\[
	\widetilde{\varphi}\left(  x\right)  =\varphi^{\ast}\left(  x\right)
	-\int_{0}^{\pi}\varphi^{\ast}\left(  x\right)  dx,
	\]
	clearly, $\widetilde{\varphi}\in H_{\ast}^{1}\left(  0,\pi\right)  .$ Taking
	$V^{\ast}=\left(  0,\widetilde
	{\varphi},0\right)  $ in (\ref{d}) we get%
	\begin{equation}
		\int_{0}^{\pi}\left(  \alpha\varphi_{x}+\beta\psi_{x}\right)  \widetilde
		{\varphi}_{x}dx=-\int_{0}^{\pi}\left(  g_{1}+bu_{x}+\alpha_{1}\varphi
		+\alpha_{3}\psi\right)  \widetilde{\varphi}dx,\,\forall\widetilde{\varphi}\in
		H_{\ast}^{1},\label{D}%
	\end{equation}
	which means that%
	\begin{equation}
		\alpha\varphi+\beta\psi\in H^{2}\left(  0,\pi\right),\label{e}%
	\end{equation}
with
\[\alpha\varphi_{xx}+\beta\psi_{xx}=  g_{1}+bu_{x}+\alpha_{1}\varphi
+\alpha_{3}\psi.\]
	Similarly, we obtain%
	\begin{equation}
		\beta\varphi+\gamma\psi\in H^{2}\left(  0,\pi\right) \label{f}%
	\end{equation}
\[\beta\varphi_{xx}+\gamma\psi_{xx}=g_2+du_x+\alpha_{3}\varphi+\alpha_{2}\psi.\]
	From (\ref{e}) and (\ref{f}) we get%
	\[
	\varphi,\psi\in H^{2}\left(  0,\pi\right)  .
	\]
	To show that $\varphi$ belongs to $H_{\ast}^{2}\left(  0,\pi\right)  $ we take
	$\varphi^{\ast}\in C^{1}\left(  0,\pi\right)  $ in (\ref{D}) and define
	$\widetilde{\varphi}$ as above, then using integration by parts, we obtain,%
	\begin{equation}
		\left[  \left(  \alpha\varphi_{x}+\beta\psi_{x}\right)  \widetilde{\varphi
		}\right]  _{0}^{\pi}-\int_{0}^{\pi}\left(  \alpha\varphi_{xx}+\beta\psi
		_{xx}-g_{1}-bu_{x}-\alpha_{1}\varphi-\alpha_{3}\psi\right)  \widetilde
		{\varphi}dx=0,\,\forall\widetilde{\varphi}\in H_{\ast}^{1}.\label{F}%
	\end{equation}
	First, we take $\widetilde{\varphi}\in C_{0}^{1}\left(  0,\pi\right)  ,$ we
	get%
	\[
	\alpha\varphi_{xx}+\beta\psi_{xx}=g_{1}+bu_{x}+\alpha_{1}\varphi+\alpha
	_{3}\psi,\quad a.e.\;\mbox{in}\;\left(  0,\pi\right)  .
	\]
	Back to (\ref{F}), we get%
	\[
	\left(  \alpha\varphi_{x}\left(  \pi\right)  +\beta\psi_{x}\left(  \pi\right)
	\right)  \widetilde{\varphi}\left(  \pi\right)  -\left(  \alpha\varphi
	_{x}\left(  0\right)  +\beta\psi_{x}\left(  0\right)  \right)  \widetilde
	{\varphi}\left(  0\right)  =0,\,\forall\widetilde{\varphi}\in H_{\ast}^{1}.
	\]
	As $\widetilde{\varphi}$ is arbitrary in $H_{\ast}^{1}\left(  0,\pi\right)  $,
	we obtain%
	\[
	\alpha\varphi_{x}\left(  \pi\right)  +\beta\psi_{x}\left(  \pi\right)
	=0\;\mbox{and}\;\alpha\varphi_{x}\left(  0\right)  +\beta\psi_{x}\left(
	0\right)  =0.
	\]
	Similarly, we obtain%
	\[
	\beta\varphi_{x}\left(  \pi\right)  +\gamma\psi_{x}\left(  \pi\right)
	=0\;\mbox{and}\;\beta\varphi_{x}\left(  0\right)  +\gamma\psi_{x}\left(
	0\right)  =0.
	\]
	Therefore, $\varphi,\psi\in H_{\ast}^{2}\left(  0,\pi\right)  ,$ consequently
	$U\in D\left(  \mathcal{A}\right)  ,$ and $0\in\rho\left(  \mathcal{A}\right)
	.$ Moreover, using a geometric series argument we prove that $\lambda
	I-\mathcal{A}=\mathcal{A}(\lambda\mathcal{A}^{-1}-I)$ is invertible for
	$|\lambda|<\Vert\mathcal{A}^{-1}\Vert$, then $\lambda\in\rho(\mathcal{A})$,
	which completes the proof that $\mathcal{A}$ is the infinitesimal generator of
	a C$_{0}-$semigroup, then the Lumer-Phillips theorem ensures the existence of
	unique solution to the problem (\ref{A}),(\ref{A1}) and (\ref{A3}) satisfying
	the statements of Theorem \ref{TH1}.
\end{proof}

\begin{remark}
	We note that if $U_{0}\in D(\mathcal{A})$ then the solution
	$U(t)=e^{t\mathcal{A}}U_{0}\in C((0,\infty);D(\mathcal{A}))\cap C^{1}%
	((0,\infty);\mathcal{H})$ and (\ref{9}) is satisfied in $\mathcal{H}$ for
	every $t>0$. It turns out that $u,\varphi,\psi$ satisfy (\ref{A}) in the
	strong sense.
\end{remark}

If $U_{0}\in\mathcal{H}$ there exists a sequence $U_{0n}\in D(\mathcal{A})$
converging to $U_{0}$ in $\mathcal{H}$. Accordingly, there exists a sequence
of solutions $U_{n}(t)=e^{t\mathcal{A}}U_{0n}$ such that $u_{n},\varphi
_{n},\psi_{n}$ satisfy (\ref{A}) in $L^{2}$ for every $t>0$, and for any
$T>0$, $u_{n}\rightarrow u$ in $C((0,T),H_{0}^{1})\cap C^{1}((0,T);L^{2})$,
$\varphi_{n}\rightarrow\varphi$ and $\psi_{n}\rightarrow\psi$ in
$C((0,T),H_{\ast}^{1})\cap C^{1}((0,T);L^{2})$. Therefore, if we multiply the
equations of (\ref{A}) for $u_{n},\varphi_{n},\psi_{n}$ by $u^{\ast}\in
H_{0}^{1}$ and $\varphi^{\ast},\psi^{\ast}\in H_{\ast}^{1}$, respectively,
then integrate by parts with respect to $x$ and integrate with respect to $t$,
finally passing to the limit, we find that $u,\varphi$ and $\psi$ are weak
solutions to the variational form of system (\ref{A}).

\subsection{Exponential stability}

In the present section we tackle the main objective of this paper, that is the
prove of the exponential decay of the solution of (\ref{A}). First we
introduce the two following constants%
\[
\chi_{0}=\left(  \frac{\mu\kappa_{1}}{\rho}-\alpha\right)  \left(  \frac
{\mu\kappa_{2}}{\rho}-\gamma\right)  -\beta^{2},
\]
and%
\[
\chi_{1}=d^{2}\left(  \frac{\mu\kappa_{1}}{\rho}-\alpha\right)  +b^{2}\left(
\frac{\mu\kappa_{2}}{\rho}-\gamma\right)  +2bd\beta.
\]
Our stability result reads as follow:

\begin{theorem}
	\label{TH2}Let\emph{ }$\left(  u,\varphi,\psi\right)  $ be a solution of
	problem \emph{(}\ref{A}) with boundary conditions (\ref{A3}). Assuming that%
	\begin{equation}
		\chi_{0}=0\hspace{0.05in}\mbox{and }\hspace{0.05in}\chi_{1}\neq0. \label{H}%
	\end{equation}
	Then the energy functional\emph{ }$E\left(  t\right)  $ defined by
	(\ref{E})\ satisfies%
	\begin{equation}
		E\left(  t\right)  \leq\lambda e^{-\xi t},\mbox{ }\forall t\geq0, \label{02}%
	\end{equation}
	where $\lambda$\ and $\xi$\ are two positive constants.
\end{theorem}

\begin{remark}
	The hypothesis (\ref{H}) is equivalent to the following:
	
	There exist two constants $\sigma,\omega\in%
	\mathbb{R}
	^{\ast},$ such that%
	\begin{equation}%
		\begin{array}
			[c]{cc}%
			\dfrac{\mu}{\rho}=\dfrac{\sigma\alpha+\omega\beta}{\sigma\kappa_{1}}%
			=\dfrac{\sigma\beta+\omega\gamma}{\omega\kappa_{2}}, & \hspace{0.04in}%
			\mbox{if }\hspace{0.04in}\beta\neq0,\\
			\left(  \dfrac{\mu}{\rho}=\dfrac{\alpha}{\kappa_{1}}\hspace{0.04in}%
			\mbox{and }\hspace{0.04in}b\neq0\right)  \hspace{0.04in}\mbox{or }\hspace
			{0.04in}\left(  \dfrac{\mu}{\rho}=\dfrac{\gamma}{\kappa_{2}}\hspace
			{0.04in}\mbox{and }\hspace{0.04in}d\neq0\right)  , & \hspace{0.04in}%
			\mbox{if }\hspace{0.04in}\beta=0.
		\end{array}
		\label{H'}%
	\end{equation}
	It is clear that in the case where $\beta\neq0$\ and $\sigma=b,\omega=d$ solve
	(\ref{H'}), then $\alpha,\beta,\gamma,\mu,\rho,\kappa_{1}$ and $\kappa_{2}$
	solve (\ref{H}).
\end{remark}

The proof of Theorem \ref{TH2}, will be established through several lemmas.

\begin{lemma}
	For $\left(  u,\varphi,\psi\right)  $ solution of (\ref{A}), there exist
	positive constants $\hat{\alpha},\hat{\gamma},\hat{\alpha}_{1}$and
	$\hat{\alpha}_{2},$ such that the functional%
	\begin{align*}
		F_{1}\left(  t\right)   &  =\kappa_{1}\int_{0}^{1}\varphi_{t}\varphi
		dx+\kappa_{2}\int_{0}^{1}\psi_{t}\psi dx+\frac{\tau_{1}}{2}\int_{0}^{1}%
		\varphi^{2}dx+\frac{\tau_{4}}{2}\int_{0}^{1}\psi^{2}dx\\
		&  -\frac{\rho}{\mu}\int_{0}^{1}u_{t}\left(  \int_{0}^{x}\left(
		b\varphi+d\psi\right)  \left(  y\right)  dy\right)  dx
	\end{align*}
	satisfies for any $\delta>0$, the estimate%
	\begin{align}
		F_{1}^{\prime}\left(  t\right)   &  \leq-\hat{\alpha}\int_{0}^{1}\varphi
		_{x}^{2}dx-\hat{\gamma}\int_{0}^{1}\psi_{x}^{2}dx-\frac{\hat{\alpha}_{1}}%
		{2}\int_{0}^{1}\varphi^{2}dx-\frac{\hat{\alpha}_{2}}{2}\int_{0}^{1}\psi
		^{2}dx\nonumber\\
		&  +\delta\int_{0}^{1}u_{t}^{2}dx+m_{\delta}\int_{0}^{1}\varphi_{t}%
		^{2}dx+m_{\delta}\int_{0}^{1}\psi_{t}^{2}dx. \label{4}%
	\end{align}
	
\end{lemma}

\begin{proof}
	The differentiation of $F_{1}\left(  t\right)  $ gives%
	\begin{align*}
		F_{1}^{\prime}\left(  t\right)   &  =\kappa_{1}\int_{0}^{1}\varphi_{tt}\varphi
		dx+\kappa_{1}\int_{0}^{1}\varphi_{t}^{2}dx+\kappa_{2}\int_{0}^{1}\psi_{tt}\psi
		dx+\kappa_{2}\int_{0}^{1}\psi_{t}^{2}dx\\
		&  +\tau_{1}\int_{0}^{1}\varphi\varphi_{t}dx+\tau_{4}\int_{0}^{1}\psi_{t}\psi
		dx-\frac{\rho}{\mu}\int_{0}^{1}u_{tt}\left(  \int_{0}^{x}\left(
		b\varphi+d\psi\right)  \left(  y\right)  dy\right)  dx\\
		&  -\frac{\rho}{\mu}\int_{0}^{1}u_{t}\left(  \int_{0}^{x}\left(
		b\varphi+d\psi\right)  _{t}\left(  y\right)  dy\right)  dx.
	\end{align*}
	By exploiting the equations of (\ref{A}) and using integration by parts, we
	get
	\begin{align*}
		F_{1}^{\prime}\left(  t\right)   &  =-\alpha\int_{0}^{1}\varphi_{x}%
		^{2}dx-2\beta\int_{0}^{1}\psi_{x}\varphi_{x}dx-\gamma\int_{0}^{1}\psi_{x}%
		^{2}dx\\
		&  -\left(  \alpha_{1}-\frac{b^{2}}{\mu}\right)  \int_{0}^{1}\varphi
		^{2}dx-2\left(  \alpha_{3}-\frac{bd}{\mu}\right)  \int_{0}^{1}\psi\varphi
		dx-\left(  \alpha_{2}-\frac{d^{2}}{\mu}\right)  \int_{0}^{1}\psi^{2}dx\\
		&  +\kappa_{1}\int_{0}^{1}\varphi_{t}^{2}dx+\kappa_{2}\int_{0}^{1}\psi_{t}%
		^{2}dx-\tau_{2}\int_{0}^{1}\psi_{t}\varphi dx-\tau_{3}\int_{0}^{1}\varphi
		_{t}\psi dx\\
		&  -\frac{\rho}{\mu}\int_{0}^{1}u_{t}\left(  \int_{0}^{x}\left(  b\varphi
		_{t}+d\psi_{t}\right)  \left(  y\right)  dy\right)  dx.
	\end{align*}
	Then, using Young's and Cauchy Schwarz inequalities, we obtain
	\begin{align*}
		F_{1}^{\prime}\left(  t\right)   &  \leq-\left(  \alpha-\beta\varepsilon
		\right)  \int_{0}^{1}\varphi_{x}^{2}dx-\left(  \gamma-\frac{b}{\varepsilon
		}\right)  \int_{0}^{1}\psi_{x}^{2}dx\\
		&  -\left[  \left(  \alpha_{1}-\frac{b^{2}}{\mu}\right)  -\left(  \alpha
		_{3}-\frac{bd}{\mu}\right)  \eta-\epsilon\right]  \int_{0}^{1}\varphi^{2}dx\\
		&  -\left[  \left(  \alpha_{2}-\frac{d^{2}}{\mu}\right)  -\frac{1}{\eta
		}\left(  \alpha_{3}-\frac{bd}{\mu}\right)  -\epsilon\right]  \int_{0}^{1}%
		\psi^{2}dx\\
		&  +m\left(  1+\frac{1}{\epsilon}+\frac{1}{\delta}\right)  \int_{0}^{1}%
		\varphi_{t}^{2}dx+m\left(  1+\frac{1}{\epsilon}+\frac{1}{\delta}\right)
		\int_{0}^{1}\psi_{t}^{2}dx+\delta\int_{0}^{1}u_{t}^{2}dx,
	\end{align*}
	for any $\varepsilon,\eta,\epsilon,\delta>0.$
	
	First, by virtue of (\ref{C2}), we can choose $\varepsilon>0$ such that%
	\[
	\hat{\alpha}=\alpha-\beta\varepsilon>0,\mbox{and }\,\,\,\hat{\gamma}%
	=\gamma-\dfrac{b}{\varepsilon}>0.
	\]

	Similarly, (\ref{C1}) allows us to choose $\eta>0$ such that%
	\[
	\hat{\alpha}_{1}=\left(  \alpha_{1}-\frac{b^{2}}{\mu}\right)  -\left(
	\alpha_{3}-\frac{bd}{\mu}\right)  \eta>0,
	\]
	and%
	\[
	\hat{\alpha}_{2}=\left(  \alpha_{2}-\frac{d^{2}}{\mu}\right)  -\frac{1}{\eta
	}\left(  \alpha_{3}-\frac{bd}{\mu}\right)  >0.
	\]
	Finally, we choose $\epsilon>0$ so that%
	\[
	\hat{\alpha}_{1}-\epsilon\geq\frac{\hat{\alpha}_{1}}{2},\,\mbox{and }\,\,\hat
	{\alpha}_{2}-\epsilon\geq\frac{\hat{\alpha}_{2}}{2}.
	\]

	Consequently, the estimate (\ref{4}) follows.
\end{proof}

\begin{lemma}
	Let $\sigma$ and $\omega$ be two constants that satisfy (\ref{H'}), then the
	functional%
	\[
	F_{2}\left(  t\right)  :=\rho\int_{0}^{\pi}\left(  \left(  \sigma\alpha
	+\omega\beta\right)  \varphi_{x}+\left(  \sigma\beta+\omega\gamma\right)
	\psi_{x}\right)  u_{t}dx+\mu\int_{0}^{\pi}\left(  \sigma\kappa_{1}\varphi
	_{t}+\omega\kappa_{2}\psi_{t}\right)  u_{x}dx
	\]
	satisfies along the solution $\left(  u,\varphi,\psi\right)  ,$ the estimate%
	\begin{align}
		\left(  b\sigma+d\omega\right)  F_{2}^{\prime}\left(  t\right)    & \leq
		-\frac{\mu}{2}\left(  b\sigma+d\omega\right)  ^{2}\int_{0}^{\pi}u_{x}%
		^{2}dx\nonumber\\
		& +m\left(  \int_{0}^{1}\varphi_{x}^{2}dx+\int_{0}^{1}\psi_{x}^{2}dx+\int
		_{0}^{1}\varphi_{t}^{2}dx+\int_{0}^{1}\psi_{t}^{2}dx\right)  .\label{5}%
	\end{align}
	
\end{lemma}

\begin{proof}
	Differentiating $F_{2}\left(  t\right)  ,$ using integration by parts and
	boundary conditions (\ref{A3}) we get%
	\[
	F_{2}^{\prime}\left(  t\right)  =\rho\left(  \frac{\mu\sigma\kappa_{1}}{\rho
	}-\left(  \sigma\alpha+\omega\beta\right)  \right)  \int_{0}^{1}u_{xt}%
	\varphi_{t}dx+\rho\left(  \frac{\mu\omega\kappa_{2}}{\rho}-\left(  \sigma
	\beta+\omega\gamma\right)  \right)  \int_{0}^{1}u_{xt}\psi_{t}dx
	\]%
	\[
	-\mu\left(  b\sigma+d\omega\right)  \int_{0}^{\pi}u_{x}^{2}dx+b\left(
	\sigma\alpha+\omega\beta\right)  \int_{0}^{\pi}\varphi_{x}^{2}dx+d\left(
	\sigma\beta+\omega\gamma\right)  \int_{0}^{\pi}\psi_{x}^{2}dx
	\]%
	\[
	+\left(  \sigma\left(  d\alpha+b\beta\right)  +\omega\left(  d\beta
	+b\gamma\right)  \right)  \int_{0}^{\pi}\varphi_{x}\psi_{x}dx-\mu\left(
	\sigma\alpha_{1}+\omega\alpha_{3}\right)  \int_{0}^{\pi}\varphi u_{x}dx
	\]%
	\[
	-\mu\left(  \sigma\alpha_{3}+\omega\alpha_{2}\right)  \int_{0}^{\pi}\psi
	u_{x}-\mu\left(  \sigma\tau_{1}+\omega\tau_{3}\right)  \int_{0}^{\pi}%
	\varphi_{t}u_{x}dx-\mu\left(  \sigma\tau_{2}+\omega\tau_{4}\right)  \int
	_{0}^{\pi}\psi_{t}u_{x}dx.
	\]
	Thus, estimate (\ref{5}) follows immediately by taking into account
	(\ref{H'}), using Young's and Poincar\'{e}'s inequalities.
\end{proof}

\begin{lemma}
	Along the solution $\left(  u,\varphi,\psi\right)  $ of (\ref{A}), the
	functional%
	\[
	F_{3}\left(  t\right)  =-\rho\int_{0}^{1}u_{t}udx
	\]
	satisfies%
	\begin{equation}
		F_{3}^{\prime}\left(  t\right)  \leq-\rho\int_{0}^{1}u_{t}^{2}dx+2\mu\int
		_{0}^{1}u_{x}^{2}dx+\frac{b^{2}}{2\mu}\int_{0}^{1}\varphi^{2}dx+\frac{d^{2}%
		}{2\mu}\int_{0}^{1}\psi^{2}dx. \label{6}%
	\end{equation}
	
\end{lemma}

\begin{proof}
	Differentiating $F_{3}\left(  t\right)  $, using integration by parts and
	Young's inequality, estimate (\ref{6}) follows immediately.
\end{proof}

\textbf{End of the proof of Theorem \ref{TH2}}

At this point we define the Lyapunov functional $\mathcal{L}\left(  t\right)
$ as follows%
\[
\mathcal{L}\left(  t\right)  =NE\left(  t\right)  +N_{1}F_{1}\left(  t\right)
+N_{2}\left(  b\sigma+d\omega\right)  F_{2}\left(  t\right)  +F_{3}\left(
t\right)  ,
\]
where $N,N_{1}$ and $N_{2}$ are positive constants to be properly chosen later.

First, we have
\[
\left\vert \mathcal{L}\left(  t\right)  -NE\left(  t\right)  \right\vert \leq
N_{1}\int_{0}^{\pi}\left(  \kappa_{1}\left\vert \varphi_{t}\varphi\right\vert
+\kappa_{2}\left\vert \psi_{t}\psi\right\vert +\frac{\tau_{1}}{2}\left\vert
\varphi\right\vert ^{2}+\frac{\tau_{4}}{2}\left\vert \psi\right\vert
^{2}\right)  dx
\]%
\[
-\frac{\rho}{\mu}\int_{0}^{\pi}\left\vert u_{t}\left(  \int_{0}^{x}\left(
b\varphi+d\psi\right)  \left(  y\right)  dy\right)  \right\vert dx+\rho
\int_{0}^{1}\left\vert u_{t}u\right\vert dx
\]%
\[
+N_{2}\left\vert b\sigma+d\omega\right\vert \int_{0}^{\pi}\left(
\rho\left\vert \sigma\alpha+\omega\beta\right\vert \left\vert u_{t}\varphi
_{x}\right\vert +\rho\left\vert \sigma\beta+\omega\gamma\right\vert \left\vert
u_{t}\psi_{x}\right\vert \right)  dx
\]%
\[
+N_{2}\int_{0}^{\pi}\left(  \mu\kappa_{1}\left\vert b\varphi_{t}%
u_{x}\right\vert +\mu\kappa_{2}\left\vert d\psi_{t}u_{x}\right\vert \right)
dx.
\]
Using Young's, Cauchy Schwarz and Poincar\'{e}'s inequalities, we obtain%
\begin{align*}
	\left\vert \mathcal{L}\left(  t\right)  -NE\left(  t\right)  \right\vert  &
	\leq c_{0}\int_{0}^{1}\left(  u_{t}^{2}+\varphi_{t}^{2}+\psi_{t}^{2}+\left(
	\varphi_{x}+\psi_{x}\right)  ^{2}+\psi_{x}^{2}+\left(  u_{x}+\varphi
	+\psi\right)  ^{2}\right)  dx\\
	&  \leq cE\left(  t\right)  .
\end{align*}
Thus,%
\[
\left(  N-c\right)  E\left(  t\right)  \leq\mathcal{L}\left(  t\right)
\leq\left(  N+c\right)  E\left(  t\right)  .
\]
Secondly, substituting (\ref{B2}),(\ref{4}),(\ref{5}) and (\ref{6}) in the
expression of $\mathcal{L}^{\prime}\left(  t\right)  $ we get%
\begin{align*}
	\mathcal{L}^{\prime}\left(  t\right)   &  \leq-\left[  \frac{1}{2}\left(
	\tau_{1}-\frac{\left(  \tau_{2}+\tau_{3}\right)  ^{2}}{4\tau_{4}}\right)
	N-m_{\delta}N_{1}-mN_{2}\right]  \int_{0}^{1}\varphi_{t}^{2}dx\\
	&  -\left[  \frac{1}{2}\left(  \tau_{4}-\frac{\left(  \tau_{2}+\tau
		_{3}\right)  ^{2}}{4\tau_{1}}\right)  N-m_{\delta}N_{1}-mN_{2}\right]
	\int_{0}^{1}\psi_{t}^{2}dx\\
	&  -\mu\left(  \frac{\left(  \sigma b+\omega d\right)  ^{2}}{2}N_{2}-2\right)
	\int_{0}^{1}u_{x}^{2}dx-\left(  \rho-\delta N_{1}\right)  \int_{0}^{1}%
	u_{t}^{2}dx\\
	&  -\left(  \hat{\alpha}N_{1}-mN_{2}\right)  \int_{0}^{1}\varphi_{x}%
	^{2}dx-\left(  \hat{\gamma}N_{1}-mN_{2}\right)  \int_{0}^{1}\psi_{x}^{2}dx\\
	&  -\frac{1}{2}\left(  \hat{\alpha}_{1}N_{1}-\frac{b^{2}}{\mu}\right)
	\int_{0}^{1}\varphi^{2}dx-\frac{1}{2}\left(  \hat{\alpha}_{2}N_{1}-\frac
	{d^{2}}{\mu}\right)  \int_{0}^{1}\psi^{2}dx.
\end{align*}
Now, we have to choose the coefficients carefully. First, we take
\[
\delta=\frac{\rho}{2N_{1}}.
\]
Secondly, We choose $N_{2}$ large enough such that%
\[
\frac{\left(  \sigma b+\omega d\right)  ^{2}}{2}N_{2}-2>0.
\]
Next, we pick $N_{1}$ large enough such that%
\[
\hat{\alpha}N_{1}-m_{\delta}N_{2}>0,\text{ }\hat{\gamma}N_{1}-m_{\delta}%
N_{2}>0,
\]%
\[
\hat{\alpha}_{1}N_{1}-\frac{b^{2}}{\mu}>0,\,\mbox{and }\,\text{ }\hat{\alpha
}_{2}N_{1}-\frac{d^{2}}{\mu}>0.
\]
Finally, we take $N$ large enough such that $\mathcal{L}\left(  t\right)
\sim$ $E\left(  t\right)  $ (i.e. $N-c>0$) and%
\begin{align*}
	\frac{1}{2}\left(  \tau_{1}-\frac{\left(  \tau_{2}+\tau_{3}\right)  ^{2}%
	}{4\tau_{4}}\right)  N-mN_{1}-m_{\delta}N_{2} &  >0,\\
	\frac{1}{2}\left(  \tau_{4}-\frac{\left(  \tau_{2}+\tau_{3}\right)  ^{2}%
	}{4\tau_{1}}\right)  N-mN_{1}-m_{\delta}N_{2} &  >0.
\end{align*}
Therefore, there exist $\sigma$ and $\widetilde{\sigma}$ positives constants
such that%
\begin{align*}
	\mathcal{L}^{\prime}\left(  t\right)   &  \leq-\sigma\int_{0}^{\pi}\left(
	\varphi_{t}^{2}+\psi_{t}^{2}+u_{t}^{2}+u_{x}^{2}+\varphi_{x}^{2}+\psi_{x}%
	^{2}+\psi^{2}+\varphi^{2}\right)  dx,\\
	&  \leq-\widetilde{\sigma}E\left(  t\right)  ,\text{ \ \ \ \ }\ \forall
	t\geq0.
\end{align*}
Since $E\left(  t\right)  $ is equivalent to $\mathcal{L}\left(  t\right)  ,$
we infer that%
\[
\mathcal{L}^{\prime}\left(  t\right)  \leq-\omega\mathcal{L}\left(  t\right)
,~~\forall t\geq0,
\]
for some positive constant $\omega.$ Thus%
\[
\mathcal{L}\left(  t\right)  \leq\lambda_{1}\mathcal{L}\left(  0\right)
e^{-\omega t},~~\forall t\geq0.
\]
Using again the equivalence between $\mathcal{L}\left(  t\right)  $ and
$E\left(  t\right)  $ we conclude that%
\[
E\left(  t\right)  \leq\lambda e^{-\omega t},~~\forall t\geq0,
\]
which completes the proof of Theorem \ref{TH2}.

\begin{remark}
	The same proof is valid for the following boundary conditions%
	\[%
	\begin{array}
		[c]{cc}%
		\begin{array}
			[c]{c}%
			u_{x}\left(  t,\pi\right)  =\varphi\left(  t,\pi\right)  =\psi\left(
			t,\pi\right)  =0,\\
			u\left(  t,0\right)  =\varphi_{x}\left(  t,0\right)  =\psi_{x}\left(
			t,0\right)  =0,
		\end{array}
		& t\geq0.
	\end{array}
	\]
	
\end{remark}

\section{Lack of exponential decay}

In this section we suppose that (\ref{H}) does not hold, and prove that the
solution $\left(  u,\varphi,\psi\right)  $ of the system (\ref{A}) lacks
exponentially stability. The proof is based on the following theorem due to
Gerhart-Pr\"{u}ss-Huang \cite{Gearhart,Pruss,Huang}.

\begin{theorem}
	Let $S\left(  t\right)  =e^{\mathcal{A}t}$ be a $C_{0}-$semigroup of
	contractions on a Hilbert space $\mathcal{H}$, with infinitesimal generator
	$\mathcal{A}$. Then $S\left(  t\right)  $ is exponentially stable if and only if:
	
	\begin{enumerate}
		\item[$\cdot$] $i%
		\mathbb{R}
		\subset\rho\left(  \mathcal{A}\right)  ,$
		
		\item[$\cdot$] $\underset{\left\vert \lambda\right\vert \longrightarrow\infty
		}{\overline{\lim}}\left\Vert \left(  \lambda I-\mathcal{A}\right)
		^{-1}\right\Vert _{\mathcal{L}\left(  \mathcal{H}\right)  }<\infty.$
	\end{enumerate}
\end{theorem}

Our result of non exponential stability reads as follow.

\begin{theorem}
	\label{THG}Suppose that (\ref{H}) does not hold, then the energy associated
	with the solution $\left(  u,\varphi,\psi\right)  $ of the system (\ref{A}) is
	not exponentially stable.
\end{theorem}

\begin{proof}
	It suffices to prove that there exists a sequence $\left(  F_{n}\right)
	\subset\mathcal{H}$ with bounded norm $\left\Vert F_{n}\right\Vert <1,$ such
	that
	\begin{equation}\label{100}
	\underset{\left\vert \lambda\right\vert \longrightarrow\infty}{\overline{\lim
	}}\left\Vert \left(  \lambda I-\mathcal{A}\right)  ^{-1}F_{n}\right\Vert
	_{\mathcal{H}}=\underset{\left\vert \lambda\right\vert \longrightarrow\infty
	}{\overline{\lim}}\left\Vert U_{n}\right\Vert _{\mathcal{H}}=\infty.
	\end{equation}
	Let $\left(  U_{n}\right)  _{n\in%
		\mathbb{N}
	}\subset D\left(  \mathcal{A}\right)  $ be the solution of $\left(  \lambda
	I-\mathcal{A}\right)  U_{n}=F_{n}$, then, omitting $n$ we have
	\[%
	\begin{array}
		[c]{rl}%
		i\lambda u+v & =f_{1}\\
		i\lambda\rho v+\mu u_{xx}+b\varphi_{x}+d\psi_{x} & =\rho f_{2}\\
		i\lambda\varphi+\phi & =f_{3}\\
		i\lambda\kappa_{1}\phi+\alpha\varphi_{xx}+\beta\psi_{xx}-bu_{x}-\alpha
		_{1}\varphi-\alpha_{3}\psi-\tau_{1}\phi-\tau_{2}\chi & =\kappa_{1}f_{4}\\
		i\lambda\psi+\chi & =f_{5}\\
		i\lambda\kappa_{2}\chi+\beta\varphi_{xx}+\gamma\psi_{xx}-du_{x}-\alpha
		_{3}\varphi-\alpha_{2}\psi-\tau_{3}\phi-\tau_{4}\chi & =\kappa_{2}f_{6}.
	\end{array}
	\]
	Taking $f_{1}=f_{3}=f_{4}=f_{5}=f_{6}=0$ and $f_{2}=\dfrac{1}{\rho}\sin\left(
	n\pi x\right)  ,$ then eliminating $v,\phi$ and $\chi$ we obtain%
	\[%
	\begin{array}
		[c]{cc}%
		\lambda^{2}\rho u+\mu u_{xx}+b\varphi_{x}+d\psi_{x} & =\sin\left(  n\pi
		x\right) \\
		\lambda^{2}\kappa_{1}\varphi+\alpha\varphi_{xx}+\beta\psi_{xx}-bu_{x}-\left(
		\alpha_{1}-i\lambda\tau_{1}\right)  \varphi-\left(  \alpha_{3}-i\lambda
		\tau_{2}\right)  \psi & =0\\
		\lambda^{2}\kappa_{2}\psi+\beta\varphi_{xx}+\gamma\psi_{xx}-du_{x}-\left(
		\alpha_{3}-i\lambda\tau_{3}\right)  \varphi-\left(  \alpha_{2}-i\lambda
		\tau_{4}\right)  \psi & =0.
	\end{array}
	\]
	Taking into account the boundary conditions (\ref{A3}), we are looking for
	$\left(  u,\varphi,\psi\right)  $ of the form%
	\[
	u=A\sin\left(  n\pi x\right)  ,\,\varphi=B\cos\left(  n\pi x\right)
	,\,\psi=C\cos\left(  n\pi x\right)  .
	\]
	That is%
	\[
	\left\{
	\begin{array}
		[c]{c}%
		\left(  \rho\lambda^{2}-\mu\pi^{2}n^{2}\right)  A-bn\pi B-dn\pi C=1\\
		-b\left(  n\pi\right)  A+\left[  \kappa_{1}\lambda^{2}-\left(  n\pi\right)
		^{2}\alpha-\left(  \alpha_{1}-i\lambda\tau_{1}\right)  \right]  B-\left[
		\beta\left(  n\pi\right)  ^{2}+\left(  \alpha_{3}-i\lambda\tau_{2}\right)
		\right]  C=0\\
		-d\left(  n\pi\right)  A-\left[  \beta\left(  n\pi\right)  ^{2}+\left(
		\alpha_{3}-i\lambda\tau_{3}\right)  \right]  B+\left[  \kappa_{2}\lambda
		^{2}-\left(  n\pi\right)  ^{2}\gamma-\left(  \alpha_{2}-i\lambda\tau
		_{4}\right)  \right]  C=0,
	\end{array}
	\right.
	\]
	which can be written%
	\begin{equation}
		\left(
		\begin{array}
			[c]{ccc}%
			p_{1}\left(  \lambda\right)  & -bn\pi & -dn\pi\\
			-bn\pi & p_{2}\left(  \lambda\right)  & p_{4}\left(  \lambda\right) \\
			-dn\pi & p_{5}\left(  \lambda\right)  & p_{3}\left(  \lambda\right)
		\end{array}
		\right)  \left(
		\begin{array}
			[c]{c}%
			A\\
			B\\
			C
		\end{array}
		\right)  =\left(
		\begin{array}
			[c]{c}%
			1\\
			0\\
			0
		\end{array}
		\right)  \label{G}%
	\end{equation}
	where%
	\[
	p_{1}\left(  \lambda\right)  :=\rho\lambda^{2}-\mu\left(  \pi n\right)
	^{2},\,\,p_{2}\left(  \lambda\right)  :=\kappa_{1}\lambda^{2}-\left(
	n\pi\right)  ^{2}\alpha-\left(  \alpha_{1}-i\lambda\tau_{1}\right)  ,
	\]%
	\[
	p_{3}\left(  \lambda\right)  :=\kappa_{2}\lambda^{2}-\left(  n\pi\right)
	^{2}\gamma-\left(  \alpha_{2}-i\lambda\tau_{4}\right)  ,\;p_{4}\left(
	\lambda\right)  :=-\beta\left(  n\pi\right)  ^{2}-\left(  \alpha_{3}%
	-i\lambda\tau_{2}\right)  ,
	\]%
	\[
	p_{5}\left(  \lambda\right)  :=-\beta\left(  n\pi\right)  ^{2}-\left(
	\alpha_{3}-i\lambda\tau_{3}\right)  .
	\]
	Solving (\ref{G}) we obtain%
	\[
	A=\frac{K_{1}}{p_{1}K_{1}+K_{2}},
	\]
	where,%
	\[
	K_{1}:=p_{2}p_{3}-p_{4}p_{5},\;K_{2}:=b(n\pi)^{2}\left(  dp_{4}-bp_{3}\right)
	-d\left(  n\pi\right)  ^{2}\left(  dp_{2}-bp_{5}\right)  .
	\]
	Let $\lambda$ be such that $p_{1}\left(  \lambda\right)  =0,$ then $\left(
	n\pi\right)  ^{2}=\dfrac{\rho\lambda^{2}}{\mu}$ and
	\[
	K_{1}=\frac{\rho}{\mu}\left[  \left(  \frac{\mu\kappa_{1}}{\rho}%
	-\alpha\right)  \left(  \frac{\mu\kappa_{2}}{\rho}-\gamma\right)  -\beta
	^{2}\right]  \lambda^{4}%
	\]%
	\[
	+i\frac{\rho}{\mu}\left[  \left(  \frac{\mu\kappa_{1}}{\rho}-\alpha\right)
	\tau_{4}+\left(  \frac{\mu\kappa_{2}}{\rho}-\gamma\right)  \tau_{1}%
	+\beta\left(  \tau_{2}+\tau_{3}\right)  \right]  \lambda^{3}+K_{3},
	\]
	that is
	\[K_1=\frac{\rho}{\mu}\chi_0\lambda^4+
	+i\frac{\rho}{\mu}\left[  \left(  \frac{\mu\kappa_{1}}{\rho}-\alpha\right)
	\tau_{4}+\left(  \frac{\mu\kappa_{2}}{\rho}-\gamma\right)  \tau_{1}%
	+\beta\left(  \tau_{2}+\tau_{3}\right)  \right]  \lambda^{3}+K_{3}%
	\]
	and
	\[
	K_{2}=-\frac{\rho^{2}}{\mu^{2}}\left[  b^{2}\left(  \frac{\mu\kappa_{2}}{\rho
	}-\gamma\right)  +d^{2}\left(  \frac{\mu\kappa_{1}}{\rho}-\alpha\right)
	+2bd\beta\right]  \lambda^{4}%
	\]%
	\[
	+i\dfrac{\rho}{\mu}\left[  bd\left(  \tau_{2}+\tau_{3}\right)  -b^{2}\tau
	_{4}-d^{2}\tau_{1}\right]  \lambda^{3}+K_{4}%
	\]
	that is
	\[
	K_2=-\frac{\rho^{2}}{\mu^{2}}\chi_{1}\lambda^4+i\dfrac{\rho}{\mu}\left[  bd\left(  \tau_{2}+\tau_{3}\right)  -b^{2}\tau
	_{4}-d^{2}\tau_{1}\right]  \lambda^{3}+K_{4}%
	\]
	where $K_{3},K_{4}$ are polynomials of degree $2$ in $\lambda.$
	
	At this point we discuss three cases:
		\begin{description}
		\item[1)] Suppose that $\chi_{0}\neq0$ and $\chi_{1}\neq0,$then
		\[
		A=\frac{K_{1}}{K_{2}}\approx\frac{\mu\chi_0  }{-\rho\chi_1 }\equiv c,
		\]
		for some constant $c\neq0.$
		
		\item[2)] Suppose that $\chi_{0}=\chi_{1}=0,$ then%
		\[
		\frac{\mu\kappa_{1}}{\rho}-\alpha=-\frac{b\beta}{d},\hspace{0.05in}\frac
		{\mu\kappa_{2}}{\rho}-\gamma=-\frac{d\beta}{d}.
		\]
		Consequently
		\[
		\left(  \frac{\mu\kappa_{1}}{\rho}-\alpha\right)  \tau_{4}+\left(  \frac
		{\mu\kappa_{2}}{\rho}-\gamma\right)  \tau_{1}+\beta\left(  \tau_{2}+\tau
		_{3}\right)  \neq0,
		\]
		\[
		bd\left(  \tau_{2}+\tau_{3}\right)  -b^{2}\tau_{4}-d^{2}\tau_{1}\neq0,
		\]
		by virtue of (\ref{V}) and%
		\[
		A=\frac{K_{1}}{K_{2}}\approx\frac{\left[  \left(  \frac{\mu\kappa_{1}}{\rho
			}-\alpha\right)  \tau_{4}+\left(  \frac{\mu\kappa_{2}}{\rho}-\gamma\right)
			\tau_{1}+\beta\left(  \tau_{2}+\tau_{3}\right)  \right]  }{\left[  bd\left(
			\tau_{2}+\tau_{3}\right)  -b^{2}\tau_{4}-d^{2}\tau_{1}\right]  }\equiv c.
		\]
	\end{description}
		Therefore,
		\[
		\left\Vert U\right\Vert ^{2}\geq\rho\left\Vert v\right\Vert ^{2}=\rho
		c^{2}\left\vert \lambda\right\vert ^{2}\int_{0}^{1}\sin^{2}\left(  n\pi
		x\right)  dx=\frac{\rho c^{2}\left\vert \lambda\right\vert ^{2}}{2},
		\]
		and consequently,
		\[
		\underset{\left\vert \lambda\right\vert \longrightarrow\infty}{\lim}\left\Vert
		U\right\Vert ^{2}=\infty.
		\]
		
	\begin{description}	
		\item[3)] Suppose that $\chi_{0}\neq0$ and $\chi_{1}=0,$ then%
		\[
		A=\frac{K_{1}}{K_{2}}\approx\frac{\chi_{0}  \lambda}{-i\left[  bd\left(  \tau_{2}+\tau_{3}\right)  -b^{2}%
			\tau_{4}-d^{2}\tau_{1}\right]  }\approx c\lambda,
		\]%
		\[
		\left\Vert U\right\Vert ^{2}\geq\left\Vert u_{x}\right\Vert ^{2}=A^{2}\left(
		n\pi\right)  ^{2}\int_{0}^{1}\cos^{2}\left(  n\pi x\right)  dx=\frac
		{c^{2}\left\vert \lambda\right\vert ^{4}\mu}{2\rho}%
		\]
		and
		\[
		\underset{\left\vert \lambda\right\vert \longrightarrow\infty}{\lim}\left\Vert
		U\right\Vert ^{2}=\infty.
		\]
		
	\end{description}
	Therefore, in all cases (\ref{100}) holds and consequently, the proof of Theorem \ref{THG} is completed.
\end{proof}


\end{document}